  \documentclass[12pt]{amsproc}
  \usepackage{latexsym} 
  \usepackage[all]{xy}
  \usepackage{amsfonts} 
  \usepackage{amsthm} 
  \usepackage{amsmath} 
  \usepackage{amssymb}
  \usepackage{pifont}  
  \usepackage{enumerate}
  \xyoption{2cell}

  \def\su#1{{\sp{[#1]}}}

  \def\<{{\langle}} 
  \def\>{{\rangle}}

  \def\note#1{{}}

  \def\note#1{} 

  \def\rhom#1#2#3{{{\rm Hom}\sb{#1}(#2,#3)}}

  \def\beq{\begin{equation}} 
  \def\eeq{\end{equation}}

  \def\id{\mathrm{id}}

  \def\ut{{\otimes}} 
  \def\ot{{\otimes}}

   \def\tom{\widetilde{\chi}}

    \def\oan#1{\Omega^{#1}(A)}
    \def\oa{\Omega (A)}
     \def\oau{\Omega A}
    \def\oaun#1{\Omega^{#1}\!A}

  \newcounter{zlist}

  \newcounter{blist}

  \newcounter{rlist}

\def\stac#1{\raise-.2cm\hbox{$\stackrel{\displaystyle\otimes}{\scriptscriptstyle{#1}}$}}

\def\cten#1{\raise-.2cm\hbox{$\stackrel{\displaystyle\widehat{\otimes}}
{\scriptscriptstyle{#1}}$}}

  \def\Label#1{\label{#1}\ifmmode\llap{[#1] }\else 
  \marginpar{\smash{\hbox{\tiny [#1]}}}\fi} 
  \def\Label{\label}

  \newtheorem{proposition}{Proposition}[section]
  \newtheorem{lemma}[proposition]{Lemma} 
   
  \newtheorem{theorem}[proposition]{Theorem} 

  \theoremstyle{definition} 
  \newtheorem{definition}[proposition]{Definition}
    
  \newtheorem{example}[proposition]{Example}

  \theoremstyle{remark}

  \newcounter{c} 
   
  \newcommand{\etyk}[1]{\vspace{-7.4mm}$$\begin{equation}\Label{#1} 
  \addtocounter{c}{1}} 
  \renewcommand{\]}{\ifnum \value{c}=1 $$\else \end{equation}\fi} 
\def\ot{\otimes}

\def\KK{{\mathbb K}}

\newcommand{\Cc}{\mathcal{C}}

\def\*C{{}^*\hspace*{-1pt}{\Cc}}

\def\text#1{{\rm {\rm #1}}}

 \def\1{\mathbf{1}}

  \def\tom{\widetilde{\omega}}

\begin{document} 

\title[A note on flat connections]{A note on flat noncommutative connections} 
 \author{Tomasz Brzezi\'nski}
 \address{ Department of Mathematics, Swansea University, 
  Singleton Park, \newline\indent  Swansea SA2 8PP, U.K.} 
  \email{T.Brzezinski@swansea.ac.uk}   

    \date{September 2011} 
  \begin{abstract} 
It is proven that every flat connection or covariant derivative $\nabla$ on a left $A$-module $M$ (with respect to the universal differential calculus) induces a right $A$-module structure on $M$ so that  $\nabla$ is a bimodule connection on $M$ or $M$ is a flat differentiable bimodule. Similarly a flat hom-connection on a right $A$-module $M$ induces a compatible left $A$-action. 
\end{abstract} 
  \maketitle

\section{Introduction}
The first aim of this note is to show that every flat connection or covariant derivative $\nabla$ on a left $A$-module $M$ (with respect to the universal differential calculus) induces a right $A$-module structure on $M$ so that  $\nabla$ is a bimodule connection on $M$ \cite[Section~3.6]{Mad:int} or that $M$ is a flat differentiable $A$-bimodule  \cite[Section~2.3]{BegBrz:Ser}. The idea is to use the correspondence between flat connections in noncommutative geometry and comodules of corings \cite[Section~29]{BrzWis:cor}  and explore the recently discovered remarkable fact that the category of comodules over the Sweedler coring is a braided (more precisely, symmetric) monoidal category \cite{AgoCae:cen}. Consequently, the category of modules with  flat covariant derivatives with respect to the universal calculus is a symmetric monoidal category. This note is aimed at noncommutative geometers, hence we will try to make it self-contained for this audience and will not use any of the coring-specific terminology or techniques. The interested reader is encouraged to consult \cite{Brz:fla} for a noncommutative geometry oriented review of the correspondence between comodules and flat connections. 

The second aim of the note is to use the techniques developed in the case of connections to the case of {\em hom-connections} \cite{Brz:con}, and to show that a right $A$-module $M$ that admits a flat hom-connection with respect to the universal differential algebra over $A$ inherits a natural left $A$-action that makes $M$ an $A$-bimodule. In both cases  the existence of a connection or a hom-connection yields the existence of the unital action commuting with the original one. The flatness of a (hom-)connection is responsible for the associativity. This last observation might be recognisable by the geometers  familiar e.g.\ with Frobenius manifolds; see \cite[Chapter~1]{Man:Fro}.

The existence of a connection (with respect to the universal differential algebra) on a left $A$-module $M$ has a purely algebraic meaning of projectivity: $M$ admits a connection if and only if it is a projective $A$-module \cite[Corollary~8.2]{CunQui:alg}. Similarly, the existence of a hom-connection is tantamount with the injectivity of the module \cite[Theorem~2.2]{BrzElK:int}. Thus, from the purely algebraic point of view, this note establishes the existence of a unital (but not necessarily associative on one side) bimodule structure on any projective or injective module. If such a module satisfies additional property corresponding to flatness of the (hom-)connection, then the induced action is necessarily associative.

\section{Every module with a flat connection with respect to the universal differential structure is a flat differentiable bimodule}\label{sec.alg} 
\setcounter{equation}{0}
Let $A$ be an associative algebra with identity over a field $\KK$. Unadorned tensor product is over $\KK$. By a {\em differential algebra over $A$} we mean  a non-negatively graded differential graded algebra $(\oa, d)$ with $\oan 0 = A$.  Any algebra $A$ admits the {\em universal differential graded algebra $(\oau, d)$} over $A$ defined as follows. $\oaun n = \oaun 1\ot_A\oaun 1\ot_A \ldots \ot_A \oaun 1$, i.e.\  $\oau: = T_A(\oaun 1)$ is the tensor algebra of the $A$-bimodule $\oaun 1 = \ker \mu$, where $\mu: A\ot A\to A$ is the multiplication map. In view of the natural identification $A\ot_A M \cong A$, $a\ot_A m\mapsto am$, of left $A$-modules,  $\oaun n$ can be understood as the subspace of $A^{\ot n+1}$ consisting of those tensors that vanish upon the multiplication of any two adjacent factors, i.e.\
$$
\oaun n = \{\sum_i a_0^i \otimes \ldots \otimes a_n^i\; |\; \forall k=0,\ldots , n-1, \,
\sum_i a_0^i \otimes \ldots \otimes a^i_{k-1} \ot a^i_{k}a^i_{k+1}\ot a^i_{k+2} \ot \ldots \ot a_n^i =0\}.
$$
The multiplication in $\oau$ is given by the restriction of the concatenation 
$$
(a_0\ot \ldots \ot a_n)(b_0\ot \ldots \ot b_m) = a_0\ot \ldots a_{n-1}\ot a_nb_0\ot b_1 \ot \ldots \ot b_m.
$$
The differential is defined as $d: A\to \oaun 1$, $a\mapsto 1\ot a - a\ot 1$ and extended to the whole of $\oau$ by the graded Leibniz rule and $d^2=0$, so that
$$
d\left(\sum_i a_0^i \otimes \ldots \otimes a_n^i \right) = \sum_{k=0}^n (-1)^k \sum_i a_0^i \otimes \ldots a_{k-1}^i \ot 1 \ot a_{k}^i \otimes \ldots \ot  a_n^i.
$$
 This is the differential graded algebra we are interested in in this note.

Given a left $A$-module $M$, a {\em connection} or a {\em covariant derivative} with respect to $(\oa, d)$ is a $\KK$-linear map $\nabla: M\to \oan 1\ot_A M$ satisfying the Leibniz rule, for all $m\in M$ and $a\in A$,
\begin{equation}\label{Leibniz}
\nabla (am) = a \nabla(m) + da \ot_A m;
\end{equation} 
see e.g.\ \cite{Con:non}. The map $\nabla$ can be extended to the map $\nabla: \oan n \ot_A M\to \oan {n+1} \ot_A M$ by the graded version of the Leibniz rule \eqref{Leibniz}. $\nabla$ is said to be {\em flat} provided (the left $A$-module map) $\nabla\circ\nabla : M\to \oan 2 \ot_A M$ is identically zero. 

In the case of the universal differential algebra over $A$ one can use the standard identification $A\ot A\ot_A M\cong A\ot M$ to view a connection in $M$ as a map
$$
\nabla: M\to A\ot M, \qquad \nabla(m) =  \sum m\su A\ot m\su M.
$$
The latter expression is a notation which proves useful in explicit calculations. The indices are meant to remind the reader that the first leg is in $A$ and the second in $M$, and the summation sign indicates that we are not dealing with simple tensors but finite sums. Exploring the definition of $(\oau, d)$  and using the above notation one finds that $\nabla$ is a connection provided
\begin{subequations}\label{con}
\begin{gather}
\sum m \su A m\su M = 0, \label{con1}\\
\sum (am)\su A \ot (am)\su M = \sum am\su A\ot m\su M + 1\ot am - a\ot m, \label{con2}
\end{gather}
\end{subequations}
for all $a\in A$, $m\in M$. The first equation encodes the fact that the image of $\nabla$ is in $\oaun 1 \ot_A M$ which can be identified with the kernel of the action $A\ot M\to M$. The second equation is the Leibniz rule. A connection $\nabla$ is flat provided
\begin{equation} \label{flat}
 \sum 1\ot m\su A\ot m\su M = \sum m\su A \ot  m\su M\su A \ot m\su M \su M := \sum m\su A \ot \nabla (m\su M).
\end{equation}

\begin{lemma}\label{lem.act}
If a left $A$-module $M$ has a flat connection $\nabla$ with respect to $(\oau, d)$, then $M$ is an $A$-bimodule with the right action defined through the left action by
\begin{equation}\label{act}
ma
 := am - \sum m\su A a m\su M,
\end{equation}
for all $a\in A$ and $m\in M$.
\end{lemma}
\begin{proof}
This is contained in \cite{AgoCae:cen} and can be revealed through the identification of modules with a flat connection with (left) comodules of the particular coring known as the Sweedler coring. However, as we do not assume that the reader is familiar with corings, we prove this lemma directly.

The unitality of the right $A$-action, i.e.\ that $m1=1$ follows by the unitality of the left $A$-action and equation \eqref{con1}. To prove the associativity, take any $a,b\in A$ and $m\in M$ and compute
\begin{equation*}
\begin{split}
(ma)b & = (am)b - \sum (m\su Aam\su M)b \\
&= bam - \sum(am)\su Ab(am)\su M -
\sum bm\su Aam\su M + \sum (m\su Aam\su M)\su Ab (m\su Aam\su M)\su M\\
&= -\sum am\su A bm \su M +abm + \sum m\su Aam\su M\su A bm\su M\su M - m\su Aabm\su M\\
&= abm +\sum am\su A b m\su M = m(ab),
\end{split}
\end{equation*}
where the third equation follows by \eqref{con2} and the fourth one by \eqref{flat}. Finally, in order to check the compatibility between left and right $A$-actions, take any $a,b\in A$ and $m\in M$, and, using the Leibniz rule \eqref{con1}, compute
\begin{equation*}
\begin{split}
(am)b &= bam - \sum (am)\su A b(am)\su M \\
&= bam - \sum am\su Abm\su M - bam +abm = a(mb),
\end{split}
\end{equation*}
as required.
\end{proof}

In view of Lemma~\ref{lem.act} one can take tensor product $M\ot_AN$ of any two left $A$-modules with flat connections. 

\begin{lemma}\label{lem.braid}
Given two left $A$-modules $M$, $N$ with flat connections  with respect to $(\oau, d)$, the following map
\begin{equation}\label{braid}
c_{M,N} : M\ot_A N \to N\ot_A M , \qquad m\ot_A n \mapsto n\ot_A m - \sum m\su A n \ot_A m\su M,
\end{equation}
is an isomorphism of $A$-bimodules. 
\end{lemma}
\begin{proof}
This again can be read off \cite{AgoCae:cen} and again we will give an explicit proof. As a priori it is not clear that the map $c_{M,N}$ is well-defined it appears prudent  to consider first the lifting
$$
\tilde{c}_{M,N} : M\ot N \to N\ot M , \qquad m\ot n \mapsto n\ot m - \sum m\su A n \ot m\su M.
$$
The left $A$-linearity of $\tilde{c}_{M,N}$ is proven as follows, for all $a\in A$ and $m\in M$,
\begin{equation*}
\begin{split}
\tilde{c}_{M,N}(am\ot n) &= n\ot am - \sum (am)\su A n \ot (am)\su M \\
&= n\ot am - \sum  am\su An\ot m\su M - n\ot am + an\ot m = a\tilde{c}_{M,N}(m\ot n),
\end{split}
\end{equation*}
where the second equality follows by the Leibniz rule \eqref{con2}. Using the definition of the right $A$-action  on $M$ \eqref{act} and the left $A$-linearity of $\tilde{c}_{M,N}$ (to derive the second equality) one can compute
\begin{equation*}
\begin{split}
\tilde{c}_{M,N}(ma\ot n) &= \tilde{c}_{M,N}(am\ot n) -\tilde{c}_{M,N}( \sum m\su Aa m\su M \ot n) \\
&= an\ot m - \sum am\su An\ot m\su M -   \sum m\su Aan\ot m\su M \\
& ~\hspace{.7in} + \sum m\su Aam\su M\su An\ot m\su M\su M\\
&= an\ot m - \sum am\su An\ot m\su M  - \sum m\su Aan\ot m\su M\\
& ~\hspace{.7in}+ \sum am\su An\ot m\su M \\
& = \tilde{c}_{M,N}(m\ot an) .
\end{split}
\end{equation*}
The penultimate equality follows by \eqref{flat}. These properties of $\tilde{c}_{M,N}$  imply that the map $\tilde{c}_{M,N}$ gives rise to the  left $A$-linear map $c_{M,N}$ as described in \eqref{braid}. 
Next, using the definition of right $A$-action on $M$, the flatness condition \eqref{flat} and \eqref{braid} one can compute
\begin{equation*}
\begin{split}
c_{M,N}(m\ot _A n) a &= n\ot_A ma -\sum m\su A n\ot_A m\su Ma\\
&= n\ot_A a m - \sum n\ot_A m\su Aam\su M - \sum m\su An\ot_A am\su M\\
& ~\hspace{.8in}+\sum m\su An\ot_A m\su M\su Aam\su M \su M\\
&= na \ot_A m - \sum n\ot_A m\su Aam\su M - \sum m\su Ana\ot_A m\su M\\
& ~\hspace{.8in}+\sum n\ot_A m\su Aam\su M\\
&= c_{M,N}(m\ot_A na).
\end{split}
\end{equation*}
This means that $c_{M,N}$  is a right $A$-linear map as required. Finally, $c_{N,M}$ is the inverse of $c_{M,N}$ as
\begin{equation*}
\begin{split}
c_{N,M} & \circ c_{M,N}(m\ot_A n) = m\ot_A n - \sum n\su A m\ot_A n \su N - \sum m\su M \ot_A m\su An \\
& ~\hspace{.75in} + \sum (m\su A n)\su A m\su M \ot_A (m\su A n)\su N\\
&=  m\ot_A n - \sum n\su A m\ot_A n \su N - \sum m\su M \ot_Am\su A n 
 + \sum m\su A n\su A m\su M \ot_A n\su N\\
& ~\hspace{.75in} + \sum m\su M \ot_Am\su A n 
- \sum m\su A m\su M\ot_A n \\
&=  m\ot_A n - \sum  m n\su A\ot_A n \su N = m\ot_A n,
\end{split}
\end{equation*}
where the Leibniz rule \eqref{con1} is used in derivation of the second equality. The remaining equalities follow by properties \eqref{con2} (for both connections on $M$ and $N$) and the definition of the right $A$-action on $N$ (remember that the tensor product is over $A$, and terms marked $m\su A$, $n\su A$ are elements of $A$).
\end{proof}

In fact the results of \cite{AgoCae:cen} assert that the maps $c_{M,N}$ define a braiding (more precisely,  symmetry, since $c_{M,N}\circ c_{N,M} = \id$) for the category of flat connections (with respect to the universal differential algebra). It is worth to keep this in mind, however we make no use of this fact in the present note.

The following definition combines notions introduced in \cite{Mou:lin},\cite{DubMad:lin} (bimodule connections) and in \cite{BegBrz:Ser} (flat differentiable bimodules).

\begin{definition}\label{def.main}
Let $M$ be an $A$-bimodule and fix a differential graded algebra $(\oa ,d)$ over $A$. Denote the product  $\oan 1\ot_A \oan 1 \to \oan 2$ by $\mu_{\oa}$.

A connection $\nabla : M\to \oan 1\ot_A M$ on (the left $A$-module) $M$ is called a {\em bimodule connection} if there exists $\sigma^1 : M\ot_A \oan 1\to \oan 1 \ot_A M$ such that
\begin{equation}\label{bim.con}
\nabla (ma) = \nabla(m) a + \sigma^1(m\ot_A da).
\end{equation}

$M$ together with a flat bimodule connection $(\nabla, \sigma^1)$ is called a {\em flat differentiable bimodule} if there exists a map $\sigma^2 : M\ot_A \oan 2\to \oan 2 \ot_A M$ extending $\sigma^1$ so that the following equality
\begin{equation}\label{flat.bim}
(\mu_{\oa} \ot_A)\circ (\id \ot_A \sigma^1)\circ (\sigma^1 \ot_A \id) = \sigma^2 \circ (\id \ot_A \mu_{\oa})
\end{equation}
holds on $M\ot_A \oan 1\ot_A \oan 1$.
\end{definition}

It is worth pointing out \cite{DubMas:fir} that if $M$ is an $A$-bimodule and $\nabla$ is a connection (in the left $A$-module $M$) with respect to the universal differential algebra over $A$, then it is automatically a bimodule connection. Since the space $  M \ot_A\oaun 1$ can be identified with the kernel of the right $A$-action, i.e.,
$$
M \ot_A\oaun 1 \cong \{\sum_i m_i\ot a_i \in M\ot A \; |\; \sum_i m_ia_i =0\},
$$
there is always a well-defined map
\begin{equation}\label{sigu}
\sigma^ 1 : M\ot _A\oaun 1 \to \oaun 1 \ot_A M , \qquad \sum_i m_i\ot a_i \mapsto - \sum_i \nabla(m_i)a_i,
\end{equation}
known as the {\em right universal sumbol} of $\nabla$ \cite{DubMas:fir}. Evidently, $\sigma^1$ is a right $A$-module map. Its left $A$-linearity follows by the Leibniz rule \eqref{con2} coupled with the defining property $ \sum_i m_ia_i =0$. Since $da = 1\ot a -a\ot 1$, one easily checks that the condition \eqref{bim.con} is satisfied. Furthermore,  the multiplication in the universal differential algebra over $A$ is provided by the tensor product and $\oaun 2 = \oaun 1\ot_A \oaun 1$, hence the equality \eqref{flat.bim} takes in this case simpler form:
\begin{equation}\label{flat.bim.uni}
(\id \ot_A \sigma^1)\circ (\sigma^1 \ot_A \id) = \sigma^2 .
\end{equation}
Since $\sigma^1$ is a bimodule map, this need not to be treated as a condition for $\sigma^2$ but rather as the definition of $\sigma^2$. Therefore, if $\nabla$ is a flat connection with respect to $\oau$, then $M$ is a flat differentiable bimodule. In view of this, the following theorem is not so much a statement about the existence but rather about the particular form of the maps $\sigma^1$ and $\sigma^2$. 

\begin{theorem}\label{thm.main}
Let $M$ be a left $A$-module with a flat connection $\nabla$ with respect to the universal differential algebra over $A$. Then $M$ is a flat differentiable bimodule with right $A$-action \eqref{act}, $\sigma^1 = c_{M,\oaun 1}$ and $\sigma^2 = c_{M, \oaun 2}$, where flat connections on the $\oaun n$ (needed for the defintions \eqref{braid}) are provided by $d$. 
\end{theorem}
\begin{proof}
By Lemma~\ref{lem.act}, $M$ is an $A$-bimodule and maps $\sigma^1$, $\sigma^2$ are well-defined bimodule (iso)morphisms by Lemma~\ref{lem.braid}. We start by checking that $\nabla$ is a bimodule connection on $M$. Take any $a\in A$ and $m\in M$ and compute:
\begin{equation*}
\begin{split}
\nabla(ma) &= \nabla( am) - \nabla (\sum m\su Aam\su M)\\
&= \sum am\su A\ot m\su M + 1\ot am -a\ot m -\sum m\su Aa m\su M\su A\ot m\su M\su M\\
& ~\hspace{1.4in} - \sum 1\ot m\su Aam\su M +\sum m\su Aa \ot m\su M\\
&= 1\ot am -a\ot m - \sum 1\ot m\su Aam\su M +\sum m\su Aa \ot m\su M.
\end{split}
\end{equation*}
The first equality is the consequence of the definition of right $A$-action on $M$, second follows by the Leibniz rule \eqref{con1} and the third one by the flatness of $\nabla$ \eqref{flat}. On the other hand,
\begin{equation*}
\begin{split}
\nabla(m)a &+ c_{M,\oaun 1}(m\ot_A da) = \sum m\su A\ot m\su Ma +da\ot_A m - \sum m\su A da \ot _A m\su M\\
&= \sum m\su A\ot am\su M - \sum m\su A\ot  m\su M\su A am\su M\su M +1\ot am\\
&~\hspace{1.4in} -a\ot m- \sum m\su A\ot am\su M + \sum m\su Aa\ot m\su M\\
&= - \sum 1\ot m\su Aam\su M + 1\ot am
-a\ot m  + \sum m\su Aa\ot m\su M
= \nabla (ma).
\end{split}
\end{equation*}
The first equality is simply the defintiion of the map $c_{M,\oaun 1}$ \eqref{braid}, while the second follows by the definitions of the right $A$-action \eqref{act} and the universal differential $d$, and by the standard identification $A\ot A\ot_A M\cong A\ot M$. The penultimate equality follows by \eqref{flat}. Therefore, $\nabla$ is a left bimodule connection, as stated.

Equation \eqref{flat.bim.uni} is checked by taking any $m\in M$ and $\omega, \tom \in \oaun 1$ and computing (tensor product over $A$)
\begin{equation*}
\begin{split}
(\id \ot \sigma^1)&\circ (\sigma^1 \ot \id)(m\ot \omega\ot \tom) = (\id \ot \sigma^1)(\omega \ot m\ot\tom - \sum m\su A\omega \ot m\su M\ot \tom)\\
&= \omega \ot \tom \ot m -  \sum \omega\ot m\su A \tom \ot m\su M - \sum m\su A\omega\ot \tom \ot m\su M\\
&~\hspace{1.4in} + \sum m\su A\omega\ot m\su M\su A\tom \ot m\su M\su M\\
&= \omega \ot \tom \ot m - \sum m\su A\omega\ot \tom \ot m\su M = c_{M,\oaun 2}(m\ot \omega\ot \tom). 
\end{split}
\end{equation*}
The first two equalities follow by the definition of $\sigma^1 = c_{M,\oaun 1}$, while the second equality is a consequence of the flatness of the connection \eqref{flat}. The final equality is the definition of $c_{M,\oaun 2}$. Since $\sigma^2 = c_{M,\oaun 2}$, the compatibility condition \eqref{flat.bim.uni} is satisfied and the theorem is proven.
\end{proof}

By \cite[Propositions~2.12~\&~2.15]{BegBrz:Ser}, a flat differentiable bimodule induces an endofunctor on the category of modules with (flat) connections: If $(M,\nabla, \sigma^1,\sigma^2)$ is a flat differentiable bimodule and $(N,\nabla_N)$ is a module with connection, then there is a covariant derivative on $M\ot_A N$,
$$
\nabla_{M\ot_A N} = \nabla \ot_A\id + (\sigma^1\ot_A\id)\circ (\id\ot_A \nabla_N).
$$
Furthermore, $\nabla_{M\ot_A N}$ is flat, provided $\nabla_N$ is flat. In view of Theorem~\ref{thm.main}, in the case of the universal differential algebra over $A$, given any flat connection $\nabla$ on a left $A$-module $M$ and a left connection $\nabla_N: N\to A\ot N$, $n\mapsto \sum n\su A\ot n\su N$, the map $\nabla_{M\ot_A N}: M\ot_A N\to A\ot M\ot_AN$, given by
\begin{equation*}
\begin{split}
\nabla_{M\ot_A N}(m\ot_A n) &= \sum m\su A\ot m\su M \ot_A n +\! \sum n\su A\ot m\ot_A n\su N -\! \sum 1\ot n\su Am\ot_A n\su N \\
&+ \sum m\su A\ot n\su A m\su M\ot_A n\su N -\sum m\su An\su A\ot m\su M\ot_A n\su N,
\end{split}
\end{equation*}
is a connection on $M\ot_A N$ that is flat provided $\nabla_N$ is flat. Here, the left action of $A$ on $M\ot_A N$ is $a(m\ot_A n) := am\ot_A n$.

\begin{example}\label{ex.1} Take $M=A\ot A$ and view it as a left $A$-module by the `outer' action $a(b\ot c) = ab\ot c$. The right $A$-action induced by the flat connection
$$
\nabla (a\ot b) = 1\ot a\ot b - a\ot 1\ot b + a\ot b\ot 1 - ab\ot 1\ot 1,
$$
comes out as
\begin{equation}\label{act.univ}
(a\ot b) c = ac\ot b+a[b,c]\ot 1,
\end{equation}
where $[b,c] = bc -cb$ denotes the commutator. One can easily check that the map
\begin{equation}\label{delta}
\Delta: A\ot A\to A\ot A\ot_A A\ot A, \qquad a\ot b \mapsto a\ot 1\ot_A 1\ot b,
\end{equation}
is $A$-bilinear and satisfies the associative law $(\Delta\ot_A \id)\circ \Delta = (\id\ot_A \Delta)\circ\Delta$. If $e\in A$ is an idempotent element, i.e.\ $e^2 =e$, then
$$
\Delta(e\ot 1) = e\ot 1\ot_A e\ot 1.
$$
As explained in \cite[29.3]{BrzWis:cor}, one can associate with $M$ and $e\ot 1$ a differential graded algebra $\oa$ over $A$, where  $\oa$ is the tensor algebra of the $A$-bimodule $M$ and the differential $d: \oan n \to \oan {n+1}$ is
\begin{equation}
\begin{split}
 d(m_{1}\ut_A \cdots \ut_A m_{n}) &= 
	(e\ot 1)\ot_A m_{1}\ot_A \cdots \ot_A m_{n}\! +\! (-1)^{n+1}m_{1}\ot_A \cdots \ot_A m_{n}
	\ot_A(e\ot 1) \notag \\
	&+ \sum_{i=1}^{n}(-1)^{i}m_{1}\ot_A \cdots \ot_A 
	m_{i-1}\ot_A \Delta(m_{i})\ot_A m_{i+1} 
	\ot_A \cdots\ot_A m_{n},     
    \end{split}
    \end{equation}
for all $m_1,\ldots ,m_n\in A\ot A$. Taking into account the natural identification $(A\ot A)\ot_A (A\ot A) \cong A\ot A\ot A$, one concludes that $\oan n \cong A^{\otimes n+1}$.  The forms of the action \eqref{act.univ} and the map $\Delta$ \eqref{delta} yield the following explicit description of  $\oa$. First, for any $a_0,\ldots, a_n, b\in A$, write $\psi(a_0,\ldots, a_n; b) \in A^{\otimes n+1}$ for the following sum (of $2^n$ terms)
\begin{equation*}
\begin{split}
\psi(& a_0,\ldots, a_n; b) := a_0b\ot a_1 \ot \ldots \ot  a_n \\
&+
\sum_{i=1}^n\sum_{1\leq k_1 <k_2 <\ldots < k_i \leq n} \!\!\!\!a_0[a_{k_1}, [a_{k_2}, \ldots  [a_{k_i}, b] \ldots ]]\ot a_1\ot \ldots \ot \hat{a}_{k_1} \ot \ldots \ot \hat{a}_{k_i}\ot \ldots \ot a_n,
\end{split}
\end{equation*}
where $\hat{a}_l$ means that $a_l$ should be replaced by $1$ (in the $l+1$-st position). For example,
$$
\psi(a_0,a_1, a_2; b) = a_0b\ot a_1 \ot a_2 + a_0[a_1,b]\ot 1\ot a_2 + a_0[a_2,b] \ot a_1\ot 1 + a_0[a_1,[a_2,b]]\ot 1\ot 1.
$$
With this notation at hand, the product in $\oa$ is given by
$$
(a_0\ot \ldots \ot a_n)(b_0\ot \ldots \ot b_m) = \psi(a_0,\ldots, a_n; b_0)\ot b_1\ot  \ldots \ot b_m,
$$
and the differential $d: A^{\otimes n}\to A^{\otimes n+1}$ comes out as
\begin{equation*}
\begin{split}
d(a_0 \ot \ldots \ot a_n) &= ea_0 \ot 1\ot a_1\ot \ldots \ot a_n + (-1)^{n+1}  \psi(a_0,\ldots, a_n; e)\ot 1 \\
& + \sum_{i=1}^n (-1)^i a_0\ot \ldots \ot a_{i-1}\ot 1\ot a_i \ot \ldots \ot a_n.
\end{split}
\end{equation*}
\end{example}
\begin{example}
The connection on the left $A$-module $M=A\ot A$ (with the left $A$-action as in Example~\ref{ex.1})
$$
\nabla = d\ot \id, \qquad a\ot b\mapsto 1\ot a \ot b - a\ot 1\ot b,
$$
induces the `inner' right $A$-action, 
$
(a\ot b) c = ac\ot b.
$
Rather disappointingly,  the natural `outer' right $A$-action $(a\ot b)c = a\ot bc$ on $A\ot A$, does not seem to be a manifestation of a flat connection on the left $A$-module $A\ot A$ with the `outer' left $A$-action.
\end{example}

\section{Every module admitting flat hom-connection with respect to the universal differential structure is a bimodule}
\setcounter{equation}{0}
In addition to standard connections one can also consider connections of the second kind or {\em hom-connections} introduced in \cite{Brz:con}. A right {\em hom-connection} with respect to  a differential graded algebra $(\oa, d)$ over an algebra $A$, is a pair $(M,\nabla)$, where $M$ is a right $A$-module and 
$$
\nabla :\rhom A {\oan 1} M \to M,
$$
is a $\KK$-linear map, such that, for all $f\in \rhom A {\oan 1} M$ and $a\in A$,
\begin{equation}\label{hom-con}
\nabla (fa) = \nabla (f)a + f(da).
\end{equation}
Here $\rhom A {\oan 1} M $ denotes the space of all right $A$-linear maps and  is given a right $A$-module structure by $(fa)(\omega) := f(a\omega)$, $\omega \in \oan 1$, $a\in A$.

Any hom-connection $(M,\nabla )$ can be extended to higher forms. The vector space $\bigoplus_{n=0} \rhom A {\oan n} M$ is a right module over $\oa$ with the multiplication, for all $\omega\in \oan n$, $f\in \rhom A{\oan{n+m}} M$, $\omega'\in \oan m$,
\begin{equation}\label{act.n}
f\omega (\omega') := f(\omega\omega').
\end{equation}
For any $n>0$, define 
$
\nabla_n: \rhom A {\oan {n+1}}M \to \rhom A {\oan {n}}M,
$
by
\begin{equation}\label{nablan}
\nabla_n(f)(\omega) :=  \nabla  (f\omega) + (-1)^{n+1} f(d\omega),
\end{equation}
for all $f\in \rhom A {\oan{n+1}}M$ and $\omega \in \oan{n}$. 
The map $F := \nabla \circ \nabla_1$ is a right $A$-module homomorphism which is called the {\em curvature} of  $(M,\nabla )$, and  $(M,\nabla )$ is said to be {\em flat} provided $F=0$.

\begin{proposition}\label{prop.hom}
Let $(M,\nabla)$ be a flat right hom-connection with respect to the universal differential algebra $\oau$ over $A$. Then $M$ is an $A$-bimodule with the left action, for all $a\in A$, $m\in M$
\begin{equation}\label{l.act}
am := ma + \nabla(\varphi_{m,a}),
\end{equation}
where the right $A$-linear maps $\varphi_{m,a}: \oaun 1 \to M$ are defined by
\begin{equation}\label{phi}
\varphi_{m,a}: \sum_i a_i\ot b_i \mapsto \sum_i ma_iab_i.
\end{equation}
\end{proposition}
\begin{proof}
First note that, for all $a,b\in A$ and $m\in M$,
$$
\varphi_{m,a}b = \varphi_{mb,a} \qquad \mbox{and} \qquad \varphi_{m,a}(db) = mab - mba.
$$
Hence, equation \eqref{hom-con} for  $f =\varphi_{m,b}$ takes the following form
\begin{equation}\label{hom-con-phi-0}
\nabla(\varphi_{ma,b}) = \nabla(\varphi_{m,b})a +mba -mab.
\end{equation}
In particular, for all $\omega \in \oaun 1$,
\begin{equation}\label{hom-con-phi}
\nabla(\varphi_{\varphi_{m,a}(\omega),b}) = \varphi_{\nabla(\varphi_{m,b}) ,a}(\omega) +\varphi_{mb,a}(\omega) -\varphi_{m,a}(\omega)b.
\end{equation}
Next, identify $\oaun 2$ with the subspace of $A\ot A\ot A$ defined as the intersection of kernels of multiplications $\mu\ot \id$ and $\id\ot \mu$, and define right $A$-linear maps, for all $a,b,\in A$ and $m\in M$,
\begin{equation*}\label{phiphi}
\varphi_{m,a,b}: \oaun 2 \to M, \qquad \sum_i a_i\ot b_i\ot c_i \mapsto \sum_i ma_iab_ibc_i.
\end{equation*}
Acting with $\omega = \sum_ia_i\ot b_i \in \oaun 1$ on $\varphi_{m,a,b}$ as in \eqref{act.n} one finds
\begin{equation} \label{phi.act}
\varphi_{m,a,b} \ \omega =\varphi_{\sum_i ma_iab_i, b} = \varphi_{\varphi_{m,a}(\omega), b}.
\end{equation}
Furthermore, since $d(\sum_i a_i\ot b_i) = \sum_i (1\ot a_i\ot b_i - a_i\ot 1\ot b_i+ a_i\ot b_i\ot 1)$, for all $\omega \in \oaun 1$,
\begin{equation}\label{phi.d}
\varphi_{m,a,b} \left(d\omega\right) =  \varphi_{ma, b}(\omega) -\varphi_{m,ab}(\omega) + \varphi_{m,a}(\omega)b.
\end{equation}
Combining the definition of $\nabla_1$ \eqref{nablan} with equations  \eqref{phi.act}, \eqref{phi.d} and \eqref{hom-con-phi}, one computes, for all $\omega \in \oaun 1$,
\begin{equation*}
\begin{split}
\nabla_1 (\varphi_{m,a,b}) (\omega)  &= \nabla\left(\varphi_{m,a,b} \ \omega \right) + \varphi_{m,a,b} \left(d\omega\right)\\
&= \nabla\left(\varphi_{\varphi_{m,a}(\omega), b}\right) +  \varphi_{ma, b}(\omega) -\varphi_{m,ab}(\omega) + \varphi_{m,a}(\omega)b\\
&= \varphi_{\nabla(\varphi_{m,b}) ,a}(\omega) +\varphi_{mb,a}(\omega) -\varphi_{m,a}(\omega)b +  \varphi_{ma, b}(\omega) -\varphi_{m,ab}(\omega) + \varphi_{m,a}(\omega)b \\
&= \left(\varphi_{\nabla\left(\varphi_{m, b}\right), a} +   \varphi_{mb,a}  + \varphi_{ma,b} -\varphi_{m,ab} \right)\left(\omega \right)
\end{split}
\end{equation*}
Since $\nabla$ is flat, $\nabla_1\circ\nabla =0$, so
\begin{equation}\label{hom-flat}
\nabla \left(\varphi_{\nabla\left(\varphi_{m, b}\right), a}\right) =     \nabla(\varphi_{m,ab})- \nabla( \varphi_{mb,a})  -\nabla( \varphi_{ma,b}).
\end{equation}
With \eqref{hom-con-phi-0} and \eqref{hom-flat} at hand one can prove that $M$ is a bimodule with the left $A$-action \eqref{l.act} as follows. First $1 m=m$, since $\varphi_{m,1} =0$. Next, for all $a,b\in A$ and $m\in M$,
\begin{equation*}
\begin{split}
a(bm) &= mba +\nabla(\varphi_{mb,a})  +\nabla(\varphi_{m,b})a + \nabla \left(\varphi_{\nabla\left(\varphi_{m, b}\right), a}\right)\\
&= \nabla(\varphi_{mb,a})  +\nabla(\varphi_{ma,b}) + mab + \nabla \left(\varphi_{\nabla\left(\varphi_{m, b}\right), a}\right)\\
&= mab + \nabla(\varphi_{m,ab}) = (ab)m,
\end{split}
\end{equation*}
where the second equality follows by \eqref{hom-con-phi-0} and the third one by \eqref{hom-flat}. Finally,
\begin{equation*}
\begin{split}
(am)b &= mab +\nabla(\varphi_{m,a})b = \nabla(\varphi_{mb,a}) + mba = a(mb),
\end{split}
\end{equation*}
by \eqref{hom-con-phi-0}. Therefore, $M$ is an $A$-bimodule as stated.
\end{proof}

Thus, a module with a flat hom-connection $\nabla$ is automatically an $A$-bimodule. Consequently, $\rhom A {\oaun 1} M$ is an $A$-bimodule with the left $A$-action defined by
$$
(af)(\omega) = af(\omega), \qquad \mbox{for all $a\in A$, $f\in\rhom A {\oaun 1} M$ and $\omega \in \oaun 1$.}
$$
One easily checks that, for all  $a\in A$, the map 
$$
\rhom A {\oaun 1} M \to M, \qquad f \mapsto \nabla(fa) - \nabla(f)a,
$$ 
is a left $A$-module map, hence $\nabla: \rhom A {\oaun 1} M \to M$ is a {\em first order differential operator}; see \cite[Lemma~1]{DubMas:fir}. By \cite[Theorem~1]{DubMas:fir} $\nabla$ has a left universal symbol 
$$
\sigma_L(\nabla): \oaun 1\ot_A \rhom A {\oaun 1} M \to M,
$$ 
determined by
$$
\nabla(af) = a\nabla(f) +\sigma_L(\nabla)(da \ot_A f).
$$
This symbol takes particularly simple form for the morphisms $\varphi_{m,a}$ introduced in  Proposition~\ref{prop.hom}. By noting that $b\varphi_{m,a} = \varphi_{bm,a}$, one easily computes that, for all $a,b,c\in A$ and $m\in M$,
$$
\sigma_L(\nabla)(bdc\ot_A \varphi_{m,a}) = b[a,c]m.
$$

  \section*{Acknowledgments}
  I would like to thank Michel Dubois-Violette for helpful comments.
 This research  is  partially supported by the European Commission grant 
PIRSES-GA-2008-230836 and  the Polish Government grant 1261/7.PR UE/2009/7.

\end{document}